\documentclass[leqno]{amsart}

\usepackage{stmaryrd,graphicx}
\usepackage{verbatim}
\usepackage{tikz-cd}
\usepackage{cleveref}
\usepackage[T1]{fontenc}
\usepackage{amssymb,mathrsfs,amsmath,amscd,amsthm,color}
\usepackage[all,cmtip]{xy}
\DeclareMathAlphabet{\mathpzc}{OT1}{pzc}{m}{it}

\DeclareMathOperator{\vr}{VR}

\DeclareMathOperator{\tot}{Tot}

\usepackage{amsfonts,latexsym,wasysym}
\usepackage{bbm}

\newcommand{\scra}{\mathscr{A}}
\newcommand{\scru}{\mathscr{U}}

\newcommand{\bbn}{\mathbb{N}}

\newcommand{\bbz}{\mathbb{Z}}
\newcommand{\cali}{\mathcal{I}}

\DeclareMathOperator{\im}{im}

\newcommand{\wt}{\widetilde}

\newcommand{\p}{\partial}
\newtheorem{theorem}{Theorem}

\newtheorem{proposition}[theorem]{Proposition}

\theoremstyle{definition}

\newtheorem*{question}{Question}
\newtheorem{remark}[theorem]{Remark}

\begin{document}

\title{A Homological Nerve Theorem for Open Covers}

\author[P. Gillespie]{Patrick Gillespie}
\address{University of Tennessee\\ Department of Mathematics\\
Knoxvillle, TN 37996, USA}
\email{pgilles5@vols.utk.edu}
\date{\today}

\begin{abstract}
In this note we show that a particular homological nerve theorem, which was originally proved for a finite cover of a simplicial complex by subcomplexes, also holds for an open cover of an arbitrary topological space. The motivation for this is to affirmatively answer a question about the homology groups of Vietoris metric thickenings.
\end{abstract}

\maketitle

Given a cover $\scru=\{U_i\}_{i\in\cali}$ of a topological space $X$, the \textit{nerve} of $\scru$, which we denote $\mathcal{N}(\scru)$, is a simplicial complex whose vertex set is $\cali$ and whose simplices are the finite subsets $\sigma\subset\cali$ such that the intersection $\bigcap_{i\in\sigma}U_i$ is nonempty. There are many nerve theorems, each of which relate a space $X$ with $\mathcal{N}(\scru)$, but vary on the assumptions placed on $X$ and $\scru$, as well as the conclusions drawn. 

One of the earliest examples of a nerve theorem is due to Borsuk \cite{Borsuk}. Borsuk proved that if $X$ is a finite-dimensional compact metric space and $\scra$ is a finite cover of $X$ by closed subsets of $X$ such that intersection of any subset of $\scra$ is an absolute retract, then $X$ and $\mathcal{N}(\scra)$ have the same homotopy type.

Another early example of a nerve theorem is contained in the work of Leray in \cite{Leray} and \cite{Leray2}. His work implies that if $X$ is a finite simplicial complex and $\scra$ is a finite cover of $X$ by subcomplexes such that the intersection of any subset of $\scra$ has trivial cohomology, then $H^n(X)\cong H^n(\mathcal{N}(\scra))$ for all $n$. The analogous nerve theorem for homology can be found in \cite{Brown} for example, in which K. Brown notes that the theorem is "essentially due to Leray". 

A sharper form of this homological nerve theorem is proven and used by R. Meshulam \cite{Roy}, which relaxes the condition that intersections of finite subsets of the cover $\scra$ are homologically trivial, but only shows that $H_j(X)\cong H_j(\mathcal{N}(\scra))$ for $j\leq n$ for a particular $n$.

The purpose of this note is to give a proof of the following theorem, which shows that the homological nerve theorem in \cite{Roy} holds for the case where $X$ is an arbitrary topological space and $\scru$ is an open cover of $X$.

\begin{theorem}\label{nerve}
Let $\scru =\{U_i\}_{i\in \cali}$ be an open cover of a topological space $X$, and let $N$ be the nerve of this cover. Fix an integer $k\in\bbn$. If $\wt H_j(\cap_{i\in\sigma}U_i)=0$ for all $\sigma\in N^{(k)}$ and $j\in\{0,\dots, k-\dim \sigma\}$, then 
\begin{enumerate}
\item $H_j(X)\cong H_j(N)$ for all $j\in\{0, \dots, k\}$
\item if $H_{k+1}(N)\neq 0$ then $H_{k+1}(X)\neq 0$.
\end{enumerate}
\end{theorem}

In \cite{Roy}, Meshulam assumes that $X$ is a finite simplicial complex, $\scru$ is a finite cover of $X$ by subcomplexes, and takes homology to have coefficients in a field. By making minor adjustments to Meshulam's proof, namely by using the homology---and not cohomology---spectral sequence of a cover, we may drop all of these assumptions, allowing $X$ to be an arbitrary space, $\scru$ to be an open cover of $X$, and the theorem holds for homology with arbitrary coefficient groups. The motivation for considering such a generalization is to affirmatively answer a question posed by Adams, Frick, and Virk \cite{Adams}.

\begin{question}\label{q}
If $\scru$ is a uniformly bounded open cover of a separable metric space $X$, then do the Vietoris metric thickening $\mathcal{V}^m(\scru)$ and the Vietoris complex $\mathcal{V}(\scru)$ have the same homology groups?
\end{question} 

If $\scru$ is a cover of a space $X$, the Vietoris complex $\mathcal{V}(\scru)$ is the simplicial complex whose vertex set is $X$ and whose simplices are the finite subsets of $X$ contained in some element of $\scru$. If $X$ is a metric space and $\scru$ is the collection of open subsets of $X$ with diameter less than $r>0$, then $\mathcal{V}(\scru)=\vr(X;r)$ is the Vietoris-Rips complex. The Vietoris-Rips metric thickening $\vr^m(X;r)$ was introduced in \cite{Adamaszek} and later generalized by the Vietoris metric thickening $\mathcal{V}^m(\scru)$ \cite{Adams}. The Vietoris metric thickening $\mathcal{V}^m(\scru)$ has the same underlying set as $\mathcal{V}(\scru)$, but has a metric which gives it a coarser topology than that of $\mathcal{V}(\scru)$. See \cite{Adams} for a precise definition. In general, $\mathcal{V}^m(\scru)$ is not a simplicial complex.

In \cite{Adams}, for any any $n\in\bbn$, the authors construct an open cover $\wt M_\scru$ of $\mathcal{V}^m(\scru)$ that is \textit{good up to level $n$}, that is, the intersection of any collection of at most $n$ sets from $\wt M_\scru$ is either empty or contractible. The authors of \cite{Adams} remarked that the above question could potentially be answered by using these covers in a Mayer-Vietoris spectral sequence. The argument we present does just this, as we will use Theorem \ref{nerve}, which is ultimately an application of the Mayer-Vietoris spectral sequence.

\begin{proof}[Answer to question]
Write $\scru=\{U_i\}_{i\in\cali}$ and let $\wt M_\scru$ be the open cover of $\mathcal{V}^m(\scru)$ that is good up to level $n$, whose existence is guaranteed by \cite{Adams}. Then $\wt H_j(\cap_{i\in\sigma}U_i)\cong 0$ for all $\sigma\in N^{(n-1)}$ and $j\in\bbn$, in which case Theorem \ref{nerve} then implies that $H_j(\mathcal{V}^m(\scru))\cong H_j(\mathcal{N}(\wt{M}_\scru))$ for all $j\leq n-1$. It was shown in \cite{Adams} that the $n$-skeleton of $\mathcal{N}(\wt{M}_\scru)$ coincides with that of $\mathcal{N}(\scru)$, hence $H_j(\mathcal{N}(\wt{M}_\scru))\cong H_j(\mathcal{N}(\scru))\cong H_j(\mathcal{V}(\scru))$ for all $j\leq n-1$, where the second isomorphism follows from Dowker duality \cite{Dowker}. In total, we have that $H_j(\mathcal{V}^m(\scru))\cong H_j(\mathcal{V}(\scru))$ for all $j\leq n-1$. Since $n$ was arbitrary, we conclude that $\mathcal{V}^m(\scru)$ and $\mathcal{V}(\scru)$ have isomorphic homology groups.
\end{proof}

\begin{remark}
The condition that $X$ is separable was used in \cite{Adams} to allow the authors to apply a nerve theorem of Nag\'{o}rko \cite{Nagorko} to conclude that $\mathcal{V}^m(\scru)$ and $\mathcal{V}(\scru)$ have isomorphic homotopy groups. However, it was not used in our argument. Thus we have shown that if $\scru$ is a uniformly bounded open cover of a metric space $X$, then $\mathcal{V}^m(\scru)$ and $\mathcal{V}(\scru)$ have isomorphic homology groups.
\end{remark}

\section{Mayer-Vietoris spectral sequence}

We give a brief description of the spectral sequence of a cover, which is sometimes known as the Mayer-Vietoris spectral sequence. We use $S_n(X)$ to denote the group of singular $n$-chains in a space $X$, and use $C_n(K)$ to denote the group of simplicial $n$-chains in a simplicial complex $K$. For a simplicial complex $K$, we denote the $n$-skeleton of $K$ by $K^{(n)}$ and the set of $n$-simplices of $K$ by $K_n$. Fix an open cover $\scru=\{U_i\}_{i\in\cali}$ of $X$ and to simplify the notation, let $N=\mathcal{N}(\scru)$ be the nerve of $\scru$. For a simplex $\sigma\subset\cali$ in $N$, let $U_\sigma$ denote the intersection $U_\sigma=\cap_{i\in\sigma}U_i$. 

Given the open cover $\scru$ there is an associated double complex $(A, \p', \p'')$. A \textit{double complex} is a collection of modules $\{A_{p,q}\}_{p,q\in\bbz}$ along with two collections of homomorphisms
$$\p':A_{p,q}\to A_{p-1,q}\qquad \p'':A_{p,q}\to A_{p,q-1}$$
which satisfy $\p'\p'=\p''\p''=0$ and $\p'\p''=\p''\p'$. Note that some authors instead require that $\p'\p''=-\p''\p'$. To define the double complex associated to $\scru$, set $A_{p,q}=\bigoplus_{\sigma\in N_p}S_q(U_\sigma)$. The vertical differentials $\p'':A_{p,q}\to A_{p,q-1}$ are induced by the boundary maps $\p:S_q(U_\sigma)\to S_{q-1}(U_\sigma)$, and the horizontal differentials $\p':A_{p,q}\to A_{p-1,q}$ are defined as follows. 

Fix a total order on the vertices $\{v_i\}_{i\in \cali}$ of $N$ so that each simplex of $N$ has a unique representation $\sigma = [v_0, \dots, v_p]$ for which $v_0<\dots<v_p$. Then if $\sigma = [v_0, \dots, v_p]$ is a $p$-simplex with $v_0<\dots<v_p$, define $\p_j\sigma$ to be the $(p-1)$-simplex $\p_j\sigma = [v_0, \dots, \widehat{v_j}, \dots, v_p]$ in which $\widehat{v_j}$ denotes that the vertex $v_j$ is omitted. Since $U_\sigma\subset U_{\p_j\sigma}$, we have that $\p_j$ defines an inclusion 
$$ S_q(U_\sigma)\to S_q(U_{\p_j\sigma})\to\bigoplus_{\tau\in N_{p-1}}S_q(U_\tau)$$
for each $\sigma$. These inclusions induce maps $\delta_j:\bigoplus_{\sigma\in N_p}S_q(U_\sigma)\to\bigoplus_{\tau\in N_{p-1}}S_q(U_\tau)$. We then define $\p':A_{p,q}\to A_{p-1,q}$ by setting $\p'=\sum_{i=0}^p(-1)^i \delta_i$. One may check that $\p'\p'=0$ by expanding out 
$$\p'\p'=\sum_{k=0}^p\sum_{j=0}^{p-1}(-1)^{k+j}\delta_j\delta_k$$
and using the relation $\delta_j\delta_k=\delta_{k-1}\delta_j$ if $j<k$. Note that $\p'\p''=\p''\p'$.

\begin{figure}
$$\begin{tikzcd}
  & \vdots \arrow[d]                                           & \vdots \arrow[d]                                                    & \vdots \arrow[d]                                           &                 \\
0 & \bigoplus_{\sigma\in N_0}S_2(U_\sigma) \arrow[d] \arrow[l] & \bigoplus_{\sigma\in N_1}S_2(U_\sigma) \arrow[d] \arrow[l]          & \bigoplus_{\sigma\in N_2}S_2(U_\sigma) \arrow[d] \arrow[l] & \dots \arrow[l] \\
0 & \bigoplus_{\sigma\in N_0}S_1(U_\sigma) \arrow[d] \arrow[l] & \bigoplus_{\sigma\in N_1}S_1(U_\sigma) \arrow[d, "\p''"] \arrow[l]  & \bigoplus_{\sigma\in N_2}S_1(U_\sigma) \arrow[d] \arrow[l] & \dots \arrow[l] \\
0 & \bigoplus_{\sigma\in N_0}S_0(U_\sigma) \arrow[d] \arrow[l] & \bigoplus_{\sigma\in N_1}S_0(U_\sigma) \arrow[d] \arrow[l, "\p'"'] & \bigoplus_{\sigma\in N_2}S_0(U_\sigma) \arrow[d] \arrow[l] & \dots \arrow[l] \\
  & 0                                                          & 0                                                                   & 0                                                          &                
\end{tikzcd}$$
\caption{The double complex $(A, \p', \p'')$}
\end{figure}

Given the double complex $A$, we may form the total complex $\tot A$, whose degree $n$ term is $(\tot A)_n=\bigoplus_{p+q=n}A_{p,q}$. The differential $\p$ of $\tot A$ is defined by setting $\p(c)=\p'(c)+(-1)^p\p''(c)$ for $c\in A_{p,q}$, for all $p,q\in\bbn$.

There are two natural filtrations $F'$ and $F''$ of $\tot A$ which give rise to the spectral sequences $E'$ and $E''$ respectively. The first filtration is given by $F_k'(\tot A)_n=\bigoplus_{p\leq k}A_{p,n-p}$ and the second filtration is $F_k''(\tot A)_n=\bigoplus_{q\leq k}A_{n-q,q}$. For details on the spectral sequence associated to a filtered chain complex, see \cite{Saunders}. Since $A$ is a first quadrant double complex, the spectral sequences $E'$ and $E''$ both converge to filtrations of $H_\bullet(\tot A)$. We will use the second spectral sequence to show that $H_\bullet(\tot A)\cong H_\bullet(X)$, which we then compare with the first spectral sequence to prove \Cref{nerve}.

Let $E=E''$ be the second spectral sequence of the double complex $A$. The terms of the $E^1$ page are obtained by taking the homology of $A$ with respect to $\p'$. To describe the $E^1$ page explicitly, we need the following proposition.

\begin{proposition}\label{cech}
Let $q\in\bbn$ and let $A_{\bullet,q}$ denote the chain complex
$$\begin{tikzcd}
\dots \arrow[r] & {\bigoplus_{\sigma\in N_2}S_q(U_\sigma)} \arrow[r, "\p'"] & {\bigoplus_{\sigma\in N_1}S_q(U_\sigma)} \arrow[r, "\p'"] & {\bigoplus_{\sigma\in N_0}S_q(U_\sigma)} \arrow[r] & 0.
\end{tikzcd}$$
Then $H_j(A_{\bullet, q})\cong 0$ for all $j>0$ and $H_0(A_{\bullet,q})\cong S^\scru_q(X)$.
\end{proposition}

Here $S^\scru_q(X)$ denotes the group of singular $q$-chains in $X$ whose elements $\sum_{i=0}^m n_i\sigma_i$ satisfy the condition that each singular simplex $\sigma_i$ has image in an element of $\scru$.

To prove \Cref{cech}, we provide a straightforward generalization of the proof in \cite[pg. 166]{Brown}, which assumes that $X$ is a $CW$-complex and $\scru$ is a cover of $X$ by subcomplexes. A similar proposition, but for cohomology, is proved by Frigerio and Maffei in \cite{Frigerio}. A more general version of \Cref{cech} is also proved by N. Ivanov \cite{Ivanov}.

\begin{proof}
We prove the proposition by giving an alternative characterization of the groups $\bigoplus_{\sigma\in N_p}S_q(U_\sigma)$. We begin by noting that $\bigoplus_{\sigma\in N_p}S_q(U_\sigma)$ has a basis $B$ consisting of pairs $(\sigma, f)$ where $\sigma$ is a $p$-simplex of $N$ and $f$ is a map $f:\Delta^q\to U_\sigma$. For any map $f:\Delta^q\to X$, let $N^f$ be the subcomplex of $N$ consisting of simplices $\sigma$ such that $\im(f)\subset U_\sigma$. Then there is a bijection between $B$ and the set of pairs $(f,\sigma)$ where $f$ is an arbitrary map $f:\Delta^q\to X$ and $\sigma$ is a $p$-simplex of $N^f$. This is to say that for each $p$, there exists an isomorphism 
$$\bigoplus_{\sigma\in N_p}S_q(U_\sigma)\cong \bigoplus_{f:\Delta^q\to X}C_p(N^f).$$
Moreover, these isomorphisms define an isomorphism of chain complexes
$$\begin{tikzcd}[column sep=small]
\dots \arrow[r] & \bigoplus_{\sigma\in N_2}S_q(U_\sigma) \arrow[r] \arrow[d, "\cong"] & \bigoplus_{\sigma\in N_1}S_q(U_\sigma) \arrow[r] \arrow[d, "\cong"] & \bigoplus_{\sigma\in N_0}S_q(U_\sigma) \arrow[r] \arrow[d, "\cong"] & 0 \\
\dots \arrow[r] & \bigoplus_{f:\Delta^q\to X}C_2(N^f) \arrow[r]                       & \bigoplus_{f:\Delta^q\to X}C_1(N^f) \arrow[r]                       & \bigoplus_{f:\Delta^q\to X}C_0(N^f) \arrow[r]                       & 0
\end{tikzcd}$$
where the differential in the bottom complex is induced by the boundary maps $C_p(N^f)\to C_{p-1}(N^f)$ on simplicial $p$-chains. Observe that for each $f:\Delta^q\to X$, the complex $N^f$ consists of all finite subsets of the set $\{i\in\cali: \im(f)\subset U_i\}$. Hence $N^f$ is either empty or contractible. Hence $H_p(N^f)\cong 0$ for all $p>0$ and the homology groups of the bottom chain complex (and hence the top as well) are zero at each position except $\bigoplus_{f:\Delta^q\to X}C_0(N^f)$. Here we note that $H_0(N_f)$ is either $0$ or $\bbz$, depending on whether $N_f$ is empty or not, which in turn depends on whether $f$ has image in some element of $\scru$. Thus we can take the set of maps $f:\Delta^q\to X$ which have image in some element of $\scru$ to be a basis for $\bigoplus_{f:\Delta^q\to X}H_0(N_f)$. This implies that $\bigoplus_{f:\Delta^q\to X}H_0(N_f)\cong S^\scru(X)$, completing the proof.
\end{proof}

Using \Cref{cech}, we see that the $E^1$ page of the second spectral sequence is of the form
$$\begin{tikzcd}[sep=small]
  & \vdots \arrow[d]       & \vdots \arrow[d] & \vdots \arrow[d] &       \\
0 & S^\scru_2(X) \arrow[d] & 0 \arrow[d]      & 0 \arrow[d]      & \dots \\
0 & S^\scru_1(X) \arrow[d] & 0 \arrow[d]      & 0 \arrow[d]      & \dots \\
0 & S^\scru_0(X) \arrow[d] & 0 \arrow[d]      & 0 \arrow[d]      & \dots \\
  & 0                      & 0                & 0               &      
\end{tikzcd}$$
and where the differentials are induced by $\p''$. Hence the second spectral sequence collapses at the $E^2$ page. Let $H^\scru_q(X)$ denote the $q$-th homology group of the chain complex $S^\scru_\bullet(X)$, which we note is isomorphic to $H_q(X)$ since $\scru$ is an open cover. Then we see that $E^2_{0,q}\cong H_q^\scru(X)\cong H_q(X)$ for all $q\in\bbn$ and $E^2_{p,q}\cong 0$ if $p>0$. Hence the homology of the total complex of $A$ is isomorphic to the homology of $X$, i.e. $H_\bullet(\tot A)\cong H_\bullet(X)$.

\section{Homological nerve theorem}

We are now ready to prove \Cref{nerve}. We remind the reader that our proof is simply a modification of the proof in \cite{Roy}, in which we use the homology spectral sequence of the cover $\scru$, rather than the cohomology spectral sequence.

\begin{proof}[Proof of Theorem \ref{nerve}]
Given an open cover $\scru$ of a space $X$, let $N$ be the nerve of $\scru$, let $A$ be the double complex associated to $\scru$, and let $E=E'$ be the first spectral sequence of the double complex $A$. The terms of the $E^1$ page of the first spectral sequence are obtained by taking the homology of $A$ with respect to $\p''$. Hence $E^1$ is
$$\begin{tikzcd}[sep=small]
  & \vdots                                           & \vdots                                           & \vdots                                           &                 \\
0 & \bigoplus_{\sigma\in N_0}H_2(U_\sigma) \arrow[l] & \bigoplus_{\sigma\in N_1}H_2(U_\sigma) \arrow[l] & \bigoplus_{\sigma\in N_2}H_2(U_\sigma) \arrow[l] & \dots \arrow[l] \\
0 & \bigoplus_{\sigma\in N_0}H_1(U_\sigma) \arrow[l] & \bigoplus_{\sigma\in N_1}H_1(U_\sigma) \arrow[l] & \bigoplus_{\sigma\in N_2}H_1(U_\sigma) \arrow[l] & \dots \arrow[l] \\
0 & \bigoplus_{\sigma\in N_0}H_0(U_\sigma) \arrow[l] & \bigoplus_{\sigma\in N_1}H_0(U_\sigma) \arrow[l] & \bigoplus_{\sigma\in N_2}H_0(U_\sigma) \arrow[l] & \dots \arrow[l] \\
  & 0                                                & 0                                                & 0,                                                &                
\end{tikzcd}$$
where the differentials are induced by $\p'$. For each $m\in\bbn$, there exists a surjective map $g_m:E^1_{m,0}\to C_m(N)$ defined as follows. For each $\sigma \in N$, let $P_\sigma$ denote the set of path components of $U_\sigma$, identify $H_0(U_\sigma)$ with $\bigoplus_{i\in P_\sigma}\bbz$, and let $f_\sigma:H_0(U_\sigma)\to \bbz$ be the map which sends $(n_i)_{i\in P_\sigma}$ to the sum $\sum_{i\in P_\sigma}n_i$. Then for each $m\in\bbn$, let $g_m:\bigoplus_{\sigma\in N_m}H_0(U_\sigma)\to \bigoplus_{\sigma \in N_m}\bbz$ be the map induced by the collection $\{f_\sigma:\sigma\in N_m\}$. It is not too difficult to check that the collection $\{g_m:m\in\bbn\}$ defines a morphism of chain complexes, $g:E^1_{\bullet,0}\to C_\bullet(N)$.

Under the assumption that $\wt H_j(U_\sigma)\cong 0$ for all $\sigma\in N^{(k)}$ and $j\in\{0,\dots, k-\dim\sigma\}$, we see that for all $m\leq k$, the $m$-th antidiagonal of the $E^1$ page contains only one nontrivial term, $\bigoplus_{\sigma\in N_m}H_0(U_\sigma)$. Moreover, $g_m:\bigoplus_{\sigma\in N_m}H_0(U_\sigma)\to C_m(N)$ is an isomorphism for $m\leq k$. Then from the commutative diagram
$$\begin{tikzcd}
{E^1_{k+2,0}} \arrow[d, "g_{k+2}"] \arrow[r] & {E^1_{k+1,0}} \arrow[d, "g_{k+1}"] \arrow[r] & {E^1_{k,0}} \arrow[d, "\cong"] \arrow[r] & {E^1_{k-1,0}} \arrow[d, "\cong"] \arrow[r] & \dots \\
C_{k+2}(N) \arrow[r]                            & C_{k+1}(N) \arrow[r]                            & C_{k}(N) \arrow[r]                          & C_{k-1}(N) \arrow[r]                          & \dots
\end{tikzcd}$$
we immediately see that $E^2_{m,0}\cong H_m(N)$ for all $m\leq k-1$. Using the fact that $g_{k+1}$ is surjective and $g_k, g_{k-1}$ are isomorphisms, it is also straightforward to see that $g_k$ induces an isomorphism $E^2_{k,0}\cong H_k(N)$ and $g_{k+1}$ induces a surjection $E^2_{k+1, 0}\to H_{k+1}(N)$.
Note that for $m\leq k$, the $m$-th antidiagonal of the $E^2$ page contains only one nontrivial term, $E^2_{m,0}$, and that $E^2_{p,q}\cong E^\infty_{p,q}$ for $p+q\leq k$. Hence for $0\leq m\leq k$, $H_m(\tot A)\cong E^2_{m,0}\cong H_m(N)$. Consequently, $H_m(X)\cong H_m(\tot A)\cong H_m(N)$ for all $m\leq k$. Lastly since there is a surjection $E^2_{k+1, 0}\to H_{k+1}(N)$, if $H_{k+1}(N)\neq 0$, then we must also have $E^2_{k+1, 0}\neq 0$. Since the differentials entering and leaving the term $E^r_{k+1,0}$ are zero homomorphisms for all $r\geq 2$, we have $E^\infty_{k+1,0}\cong E^2_{k+1,0}$, which in turn implies that $H_{k+1}(X)\cong H_{k+1}(\tot A)\neq 0$.
\end{proof}

\begin{remark}
The fact that $\scru$ is an open cover is only used for the isomorphism $H^\scru_\bullet(X)\cong H_\bullet(X)$ which is used to establish $H_\bullet(X)\cong H_\bullet(\tot A)$. Hence \Cref{nerve} holds slightly more generally for any space $X$ and cover $\scru$ such that $H^\scru_\bullet(X)\cong H_\bullet(X)$, for example if $\scru$ is a collection of sets whose interiors cover $X$.
\end{remark}

\section*{Acknowledgments}
The author would like to thank Henry Adams for several helpful conversations.

\end{document}